\documentclass[12 pt]{article}
\usepackage[left=3.2cm,right=3cm,
    top=3cm,bottom=3.5cm,bindingoffset=0cm]{geometry}
\usepackage{amssymb,amsmath,amsfonts,amsthm}
\usepackage{graphicx}
\usepackage{subcaption}
\usepackage{float}
 2

\newcommand{\nats}{\mathbb N}
\newcommand{\A}{\mathbb A}
\newcommand{\fr}{\mbox{fr}}
\newcommand{\Fac}{\mbox{Fac}}
\newtheorem{theorem}{Theorem}[section]
\newtheorem{corollary}[theorem]{Corollary}
\newtheorem{lemma}[theorem]{Lemma}

\newtheorem{thm}{Theorem}

\theoremstyle{definition}
\newtheorem{definition}[theorem]{Definition}
\newtheorem{example}[theorem]{Example}
\newenvironment{dedication}
  {\vspace*{1cm} 
   \itshape  
   \raggedleft   
  }
  {\par 
  }

\begin{document}
\title{\bf Open and closed factors in Arnoux-Rauzy words \thanks{This work was performed within the framework of the LABEX MILYON (ANR-10-LABX-0070) of Universit\'{e} de Lyon, within the program ``Investissements d'Avenir" (ANR-11-IDEX-0007) operated by the French National Research Agency (ANR), and has been supported by RFBS grant 18-31-00009.}}
\author{Olga Parshina\textsuperscript{1,2} and Luca Zamboni\textsuperscript{1}}
\date{}
\maketitle

\begin{center}
\textsuperscript{1}Universit\'{e} de Lyon, Universit\'{e} Lyon  1,  CNRS  UMR  5208,  Institut  Camille  Jordan,  43 boulevard du 11 novembre 1918, F69622 Villeurbanne Cedex, France\\
\textsuperscript{2}Sobolev Institute of Mathematics of the Siberian Branch of the Russian Academy of Sciences, 4 Acad. Koptyug avenue, 630090 Novosibirsk, Russia\\
\texttt{\{parshina,zamboni\}@math.univ-lyon1.fr}

\begin{dedication}
\textbf{In Memory of the Late Professor Aldo de Luca }
\end{dedication}
\end{center}

\vspace{1cm}

{\bf Abstract:} {Given a finite non-empty set $\mathbb A,$ let $\mathbb A^+$ denote the free semigroup generated by $\mathbb A$ consisting of all finite words $u_1u_2\cdots u_n$ with $u_i\in \mathbb A.$ A word $u\in \mathbb A^+$ is said to be closed if either $u\in \mathbb A$ or if $u$ is a complete first return to some factor $v\in \mathbb A^+,$ meaning $u$ contains precisely two occurrences of $v,$  one as a prefix and one as a suffix. We study the function $f_x^c:\mathbb N \rightarrow \mathbb N$ which counts the number of closed factors of each length in an infinite word $x.$ We derive an explicit formula for $f_x^c$ in case $x$ is an Arnoux-Rauzy word. As a consequence we prove that $\liminf _{n\rightarrow \infty}f_x^c(n)=+\infty.$}

\section{Introduction}
Throughout this paper we let $\nats=\{1,2,3,\dots \}$ and $\omega=\{0,1,2,3,\dots \}$ be the smallest transfinite ordinal.  Given a finite non-empty set $\A$, we let $\A^+$ denote the free semigroup generated by $\A$ consisting of all words $u_1u_2\cdots u_n$ with $u_i\in \A,$ and  $\A^\nats$ denote the set of (right) infinite words $x=x_1x_2x_3\cdots$ with $x_i\in \A.$  For each infinite word $x=x_1x_2x_3\cdots \in\A^\nats$, the factor complexity $p_x(n)$ counts the number of distinct blocks (or factors) $x_ix_{i+1}\cdots x_{i+n-1}$ of length $n$ occurring in $x.$  First introduced by Hedlund and Morse in their seminal 1938 paper
\cite{MoHe1} under the name of {\it block growth}, the factor complexity provides a useful measure of the extent of randomness of $x.$ Periodic words have bounded factor complexity while digit expansions of normal numbers
have maximal complexity. A celebrated theorem of Morse and Hedlund in \cite{MoHe1} states that every aperiodic (meaning not ultimately periodic) word contains at least $n+1$ distinct factors of each length $n$. Sturmian words are those aperiodic words of minimal factor complexity:  $p_x(n)=n+1$ for each $n\geq 1.$  

Several notions of complexity have been successfully used in the study of infinite words and their combinatorial properties. They include  Abelian complexity \cite{ABCD, CFSZ, CPZ, CovHed, Priv, RSZ},  palindrome complexity \cite{ABCD}, cyclic complexity \cite{CFSZ}, privileged complexity \cite{Priv},  group complexity \cite{CPZ} and maximal pattern complexity  \cite{KaZa} to name just a few.  
In this paper we introduce and study two new complexity functions  based on the notions of open and closed words \cite{Fici}. A word $u\in \A^+$ is said to be {\it closed} if either $u\in \A$ or if $u$ is a complete first return to some proper factor $v\in \A^+,$ meaning $u$ has precisely two occurrences of $v,$ one as a prefix and one as a suffix. If $u$ is not closed then $u$ is said to be {\it open}. Thus a word $u\in \A^+\setminus \A$  is closed if and only if it is bordered and its longest border only occurs in $u$ as a prefix and as a suffix. The longest border of a closed word is called {\it frontier}.  For example, $aabaaabaa$ is closed and its frontier is equal to $aabaa.$ In contrast, $ab$ is open as it is unbordered while $abaabbababbaaba$ is open since its frontier $aba$ occurs internally in $u.$ It is easily seen that all  privileged words \cite{Priv} are closed and hence so are all palindromic factors of rich words \cite{GJWZ}. The terminology  open and closed was first introduced by the authors in \cite{BDF} although the notion of a closed word had already been introduced earlier by A. Carpi and A. de Luca in \cite{CdL}. For a nice overview of open and closed words we refer the reader to the recent survey article by G. Fici \cite{Fici}.

For each infinite word $x\in \A^\nats$ we consider the functions $f_x^c, f_x^o:\nats \rightarrow \nats$ which count the number of closed and open factors of $x$ of each length $n\in \nats.$ In this paper we investigate the function $f_x^c$ where $x$ is an Arnoux-Rauzy word. Arnoux-Rauzy words were first introduced in \cite{AR} in the special case of a $3$-letter alphabet. They are a natural generalization of Sturmian words to alphabets of cardinality greater that two.
If $x\in \A^\nats$ is an Arnoux-Rauzy word, then $p_x(n)=(|\A|-1)n+1$ for each $n\in \nats.$ Moreover each factor $u$ of $x$ has precisely $|\A|$ distinct complete first returns in $x.$

Our main result in Theorem~\ref{main} below provides an explicit formula for the closed complexity function $f_x^c(n)$ for an Arnoux-Rauzy word $x$ on a $t$-letter alphabet $\A.$    Since for any word $x\in \A^\nats$ we have that $f_x^c(n)+f_x^o(n)=p_x(n),$ a formula for $f_x^c(n)$ also yields a formula for $f_x^o(n).$
Our formula is expressed in terms of two related sequences associated to $x.$ The first is the sequence $(b_k)_{k\geq 0}$  of the lengths of the bispecial factors 
$\varepsilon=B_0,B_1,B_2,\ldots $ of $x,$  ordered according to increasing length.  The second is 
the sequence $(p_a^{(k)})_{a\in \A}^{k\in \omega}$ where for each $k\in \omega,$ the $t$ coordinates of   $(p_a^{(k)})_{a\in \A}$ are the lengths of the $t$ first returns to $B_k$ in $x.$ Both sequences have been extensively studied in the literature. 
For each $k\in \omega,$ the coordinates of  $(p_a^{(k)})_{a\in \A}$ are coprime and each is a period of the word $B_k.$  Moreover, each $B_k$ is an extremal Fine and Wilf word, i.e.,  any word $u$ having  periods  $(p_a^{(k)})_{a\in \A}$ and of length greater than $b_k$ is a constant word, i.e., $u=a^n$ for some $n$ (see  \cite{TZ}).

\begin{thm}\label{main}
Let $x\in \A^\nats$ be an Arnoux-Rauzy word. For each $k\in \omega$  and $a\in \A$ set
$I_{k,a}=[b_k-2p_k+p_a^{(k)}+2, b_k+p_a^{(k)}]$ where  $p_k=\min_{b\in \A}\{p_b^{(k)}\}.$
Let \begin{equation}\label{formula}F(a,n) = \sum_{\substack{k\in \omega\\n\in I_{k,a}}} (d(n,I_{k,a})+1),\end{equation} where for $n \in I_{k,a},$ the quantity $d(n,I_{k,a})$ denotes the minimal distance from $n$ to the endpoints of the interval $I_{k,a}.$ Then the number of closed factors of $x$ for each length $n$ is $f^c_x(n)=\sum_{a\in \A} F(a,n)$.
\end{thm}

For each fixed $n\in \nats$ and $a\in \A,$ the sum in (\ref{formula}) is finite since it only involves those $k$ for which $n\in I_{k,a}.$

As a corollary of Theorem~\ref{main}, we show that if $x$ is an Arnoux-Rauzy word, then $\liminf f_x^c(n)=+\infty.$ 
In contrast, it follows from \cite{SS} that if $x$ is the regular paperfolding word, then 
$\liminf f_x^c(n)=0,$ in other words, for infinitely many $n,$ all factors of $x$ of length $n$ are open.

We end this section by recalling a few basic notions in combinatorics on words relevant to the paper. Throughout this text $\A$ will denote a finite non-empty set (the alphabet). For $n\in \nats,$  let $\A^n$ denote the set of all words $a_1a_2\cdots a_n$ with $a_i\in \A.$ For $u=a_1a_2\cdots a_n\in \A^n,$ we let $\overline u\in \A^n$ denote the reversal of $u,$ i.e., $\overline u=a_na_{n-1}\cdots a_1.$ Let $\A^+=\bigcup_{n\in \nats}\A^n$ denote the free semigroup generated by $\A.$ 
For $u=a_1a_2\cdots a_n\in \A^+$ the quantity $n$ is called the length of $u$ and denoted $|u|.$ We set $\A^*=\A^+\cup \{\varepsilon\}$ where $\varepsilon$ is the empty word (of length equal to $0).$ We let $\A^\nats$ denote the set of all infinite words $a_1a_2a_3\cdots$ with $a_i\in \A.$ For $x\in \A^+\cup \A^\nats$ and $v\in \A^*$ we say that $v$ is a factor of $x$ if $x=uvy$ for some $u\in \A^*$ and $y\in \A^*\cup \A^\nats.$ We let $\Fac(x)$ denote the set of all factors of $x.$ A factor $v$ of $x$ is called right (resp. left) special if $va$ and $vb$ (resp. $av$ and $bv)$ are each factors of $x$ for some choice of distinct $a,b\in \A.$ A factor which is both right and left special is said to be bispecial. 
Given factors $u$ and $v$ of $x,$ we say that $u$ is a first return to $v$ in $x$ if $uv$ is a factor of $x$ having precisely two first occurrences of $v,$ one as a prefix and one as a suffix. In this case the word $uv$ is called a complete first return to $v.$ 
 
\section{Counting closed factors in Arnoux-Rauzy words}

Throughout this section we let $\A$ denote a finite set of cardinality $t\geq 2.$
A recurrent word $x\in \A^\nats$ is called an Arnoux-Rauzy word if $x$ contains, for each $n\geq 0,$ precisely one right special factor $R_n$ of length $n$ and one left special factor $L_n$ of length $n.$ Furthermore, $R_n$ is a prefix of $t$-many distinct factors of $x$ of length $n+1$ while $L_n$ is a suffix of $t$-many distinct factors of $x$ of length $n+1.$ In particular one has $p_x(n)=(t-1)n+1$ and each factor $u$ of $x$ has precisely $t$ distinct complete first returns. In the special case of a binary alphabet, we see that $x$ is a Sturmian word. Arnoux-Rauzy words constitute a special class of episturmian words (see \cite{Ber, DJP, JP}) and hence each factor $u$ of an Arnoux-Rauzy word is (palindromically) rich, i.e., $u$ contains exactly $|u|+1$ many distinct palindromic factors (including the empty word $\varepsilon).$ We will make use of the following alternative characterisation of rich words given in \cite{DJP}.

\begin{lemma}\label{rich}[Proposition 3 in \cite{DJP}] A word $u\in \A^+$ is rich if and only if for every prefix $v$ of $u,$ the longest palindromic suffix of $v$ is uni-occurrent in $v.$  
\end{lemma}

Let us now fix an Arnoux-Rauzy word $x\in \A^\nats.$ Recall that for each length $n\in \omega$ an Arnoux-Rauzy word contains either zero or one bispecial factor of length $n.$ Let $\varepsilon =B_0, B_1, B_2, \ldots$ be the sequence of bispecial factors of $x$ ordered according to increasing length. Put $b_k=|B_k|$ so that $0=b_0<b_1<b_2<\cdots.$ We recall the following characterization of the bispecial factors $B_k$ of $x$ in terms of palindromic closures (see \cite{deL, DJP}). For each $k\in \nats$ there exists a unique $a_k\in \A$ such that $B_{k-1}a_k$ is a left special factor of $x.$ The sequence $(a_k)_{k\in \nats}$ is called the directive sequence of $x.$ It follows that $B_{k-1}a_k$ is a prefix of $B_k$ but in fact $B_k$ is the palindromic closure of $B_{k-1}a_k,$ i.e., the shortest palindrome beginning in $B_{k-1}a_k.$ More precisely, if we let $S_k$ denote the longest palindromic suffix of $B_{k-1}a_k$ and write  $B_{k-1}a_k=x_kS_k$ with $x_k\in \A^*,$ then $B_k=x_kS_k\overline {x_k}$ (see for instance Lemma 5 in \cite{deL} in case the alphabet $\A$ is binary).

\begin{lemma}\label{uni} For each $k\in \nats$ we have that $S_k$ is a uni-occurrent factor of $B_k.$ In particular $S_k$ is not a factor of $B_{k-1}.$ 
\end{lemma}

\begin{proof} Clearly $S_k$ is a factor of $B_k.$ To see that it is uni-occurrent, suppose that $S_k$ occurs more than once in $B_k.$ Since $S_k$ and $B_k$ are each palindromes and $B_k=x_kS_k\overline{x_k},$ it follows that $S_k$ occurs at least twice in $x_kS_k=B_{k-1}a_k.$ But this contradicts Lemma~\ref{rich} since $S_k$ was defined as the longest palindromic suffix of $B_{k-1}a_k.$ \end{proof}

Define $\varphi :\Fac(x)\rightarrow \omega$ by $\varphi(v)$ is the least $k\in \omega$ such that $v$ is a factor of $B_k.$ In particular $\varphi(v)=0\Leftrightarrow v=\varepsilon.$

\begin{lemma}\label{char} Let $k\in \nats$ and $v\in \A^+.$ Then $v\in \varphi^{-1}(k)$ if and only if $v$ is a factor of $B_k$ containing $S_k$ as a factor. In particular each $v\in \varphi^{-1}(k)$ is uni-occurrent in $B_k.$  
\end{lemma}

\begin{proof}Suppose $v$ is a factor of $B_k$ containing $S_k$ as a factor. Then by Lemma~\ref{uni}, $v$ is not a factor of $B_{k-1}$ and hence not a factor of any $B_j$ with $j<k.$ Hence $\varphi(v)=k.$ Conversely suppose that $\varphi(v)=k.$ Then $v$ is a factor of $B_k$ but not of $B_{k-1}.$ Since \[B_k=x_kS_k\overline{x_k}=B_{k-1}a_k\overline{x_k}=x_ka_kB_{k-1},\] it follows that $v$ must contain $S_k$ as a factor. Having established that each $v\in \varphi^{-1}(k)$ contains $S_k,$ it follows by Lemma~\ref{uni} that $v$ is uni-occurrent in $B_k.$
\end{proof}

For each $k\in \omega$ and $a\in \A,$ let $R^{(k)}_a$ denote the complete first return to $B_k$ in $x$ beginning in $B_ka$ and put $p_a^{(k)}=|R^{(k)}_a|-b_k.$ In other words $p_a^{(k)}$ is the length of the first return to $B_k$ determined by $R_a^{(k)}.$ We note that $R^{(0)}_a=a$ for each $a\in \A.$ The sequence $(p_a^{(k)})_{a\in \A}^{k\in \omega}$ is computed recursively as follows : $p_a^{(0)}=1$ for each $a \in \A.$ For $k\geq 1,$ we have $p_{a_k}^{(k)}=p_{a_k}^{(k-1)}$, and $p_b^{(k)}=p_b^{(k-1)}+p_{a_k}^{(k-1)}$ for $b\in \A\setminus \{a_k\}.$ It is easily verified by induction that  \[b_k=~\frac{\sum_{a\in \A}p_a^{(k)}-t}{t-1}.\]
For each $k\in \nats,$ we set $p_k=p_{a_k}^{(k)}.$ Since $B_k$ is a complete first return to $B_{k-1}$ beginning in $B_{k-1}a_k$ i.e.,  $R_{a_k}^{(k-1)}=B_k,$ it follows that \begin{equation}p_k=p_{a_k}^{(k)}=p_{a_k}^{(k-1)}=|R_{a_k}^{(k-1)}|-b_{k-1}=b_k-b_{k-1}.\end{equation} It follows immediately from our recursive definition of the $p_a^{(k)}$ that \[p_k=\min\{p^{(k)}_a|\,a\in \A\}.\]

\begin{lemma}\label{count}Let $k\in \nats$ and let $J_k$ denote the interval $ [b_k-2p_k+2,b_k].$ If $v\in \varphi^{-1}(k)$ then $|v|\in J_k$ and, for each $m\in J_k,$ the set $\varphi^{-1}(k)$ contains precisely $d(m,J_k)+1$ distinct words of length $m,$ where $d(m,J_k)$ is the minimal distance between $m$ and the two boundary points of the interval $J_k.$ In particular $|\varphi^{-1}(k)|=p_k^2.$ 
\end{lemma}

\begin{proof} In view of Lemma ~\ref{char} we have that $v\in \varphi^{-1}(k)$ if and only if $v$ is a factor of $B_k$ which contains $S_k$ as a subfactor 
It follows that $|S_k|\leq |v|\leq |B_k|.$ Also, since $B_k=B_{k-1}a_k\overline{x_k},$ by (2) we deduce that  $p_k=b_k-b_{k-1}=|x_k|+1.$ Furthermore, as $B_k=x_kS_k\overline{x_k}$ we have $|S_k|=|B_k|-2|x_k|=b_k-2(p_k-1)=b_k-2p_k+2.$ Hence $b_k-2p_k+2\leq |v|\leq b_k.$
Now suppose $m\in J_k.$ To see that $\varphi^{-1}(k)$ contains  $d(m,J_k)+1$ distinct words of length $m$ we simply use the fact that each $v\in \varphi^{-1}(k)$ contains $S_k$ and is uni-occurrent in $B_k$ (see Lemma~\ref{char}). Finally,
\[|\varphi^{-1}(k)|=1 +2 +\cdots +(p_k-1) +p_k +(p_k-1) +\cdots +2 +1=2\left(\frac{p_k(p_k-1)}{2}\right)+p_k=p_k^2.\]
\end{proof} 

Let $k\in \nats$ and $v\in \varphi^{-1}(k).$ 
As a consequence of Lemma~\ref{char}, there exists a unique decomposition $B_k=u_1vu_2$ with $u_1,u_2\in \A^*.$  In particular, $vu_2$ is right special in $x$ and $u_1v$ is left special in $x.$
Now suppose $u$ is a closed factor of $x$ with frontier $v.$ In particular $u$ begins and ends in $v.$ Since $x$ is recurrent and aperiodic, it follows that $vu_2$ is a prefix of $u$ and $u_1v$ is a suffix of $u,$ whence $u_1uu_2$ is a complete first return to $B_k.$ In fact, $u_1uu_2$ begins and ends in $B_k$ and does not admit other occurrences of $B_k$ for otherwise $v$ would occur in $u$ internally (meaning not as a prefix or as a suffix). Thus $u_1uu_2=R^{(k)}_a$ for some $a\in \A.$

\begin{definition} Let $u$ be a closed factor of $x$ and $a\in \A.$ We say $u$ is of {\it type} $a$ if and only if either $u=a$ or, if $u$ is closed with frontier $v\in \A^+,$ then $u_1uu_2=R^{(k)}_a$ where $k=\varphi(v)$ and $B_k=u_1vu_2.$
\end{definition}

If $u$ is a closed factor of $x$ of type $a\in \A$ and frontier $v\in \A^+,$ then 
\begin{equation}|u|-|v|=|R^{(k)}_a|-(|u_1|+|u_2|+|v|)=|R^{(k)}_a|-|B_k|=p^{(k)}_a,\end{equation}
where $k=\varphi(v).$ We observe that the equality $|u|-|v|=p^{(k)}_a$ in (3) also holds in case $u\in \A$ taking $v=\varepsilon$ and $k=0.$

Let $C(x)$ denote the set of all closed factors of $x$ and  for each $u\in C(x)$ let $\fr(u)\in \A^*$ denote its frontier. By convention we define $\fr(a)=\varepsilon$ for each $a\in \A.$ For each $k\in \omega$ and $a\in \A$ we let $C_{k,a}(x)$ denote the set of all closed factors $u$ of $x$ of type $a$ whose frontier $\fr(u)$ belongs to $\varphi^{-1}(k).$

\begin{lemma}\label{bijection} The sets $\{C_{k,a}(x)\,:\, k\in \omega,\, a\in \A\}$ define a partition of $C(x)$ and $\fr: C_{k,a}(x)\rightarrow \varphi^{-1}(k)$ is a bijection.
\end{lemma}

\begin{proof} Each closed factor $u\in C(x)$ has a unique type and its frontier $\fr(u)$ belongs to $\varphi^{-1}(k)$ for a unique value of $k\in \omega.$ Whence each closed factor $u$ of $x$ belongs to a unique $C_{k,a}(x).$ By definition, if $u\in C_{k,a}(x)$ then $\fr(u)\in \varphi^{-1}(k).$ Moreover $u$ is uniquely determined by its frontier $\fr(u)$ and its type. In fact, if $u\in C_{k,a}(x)$ then $u_1uu_2=R^{(k)}_a$ where $u_1,u_2$ are determined by the (unique) factorization $B_k=u_1\fr(u)u_2.$
This proves $\fr$ is injective. To see that $\fr$ is also surjective, let $v\in \varphi^{-1}(k).$ Then we can write $B_k=u_1vu_2$ for some $u_1,u_2\in \A^*.$
Hence $R^{(k)}_a$ begins in $u_1$ and ends in $u_2.$
It follows that $u=u_1^{-1}R^{(k)}_au_2^{-1}$ is a closed factor of $x$ of type $a$ and $\fr(u)=v.$
\end{proof}

\noindent{\it Proof of Theorem~\ref{main}.} Fix $n\in \nats.$ By Lemma~\ref{bijection} we have 
\[f^c_x(n)=|C(x)\cap \A^n|=\sum_{\substack{k\in \omega\\a \in \A}}|C_{k,a}(x)\cap \A^n|.\]

Now assume $u\in C_{k,a}\cap \A^n$ and put $v=\fr(u)\in \varphi^{-1}(k).$ Then by (3) we have that $n=|u|=|v|+p^{(k)}_a.$ By Lemma~\ref{count}, $|v|=n-p^{(k)}_a\in J_k=[b_k-2p_k+2,b_k].$ By Lemma~\ref{bijection} the number of words $u\in C_{k,a}(x)\cap\A^n$ is equal to the number of words $v\in \varphi^{-1}(k)$ of length $n-p^{(k)}_a$ which by Lemma~\ref{count} is equal to $d(n-p^{(k)}_a,J_k)+1=d(n,I_{k,a})+1$ where 
$I_{k,a}=[b_k-2p_k+p^{(k)}_a+2,b_k+p^{(k)}_a].$ This completes the proof of Theorem~\ref{main}.\qed

In case $|\A|=2,$ i.e., $x$ is Sturmian, each bispecial factor $B_k$ has precisely two first returns, the shortest one is of length $p_k$, and we let $q_k$ denote the length of the other first return. So for fixed $a\in \A$ and $k\in \nats$ we have \[p^{(k)}_a=\begin{cases} p_k,\,\,\,\mbox{if}\,\,\,a=a_k;\\q_k,\,\,\,\mbox{otherwise.}\end{cases}\]
If $a=a_k$ then $I_{k,a}=[q_k,q_k+2p_k-2]$ and if $a\neq a_k$ then $I_{k,a}=[2q_k-p_k,2q_k+p_k-2].$
Putting $P_k= [q_k,q_k+2p_k-2]$ and $Q_k=[2q_k-p_k,2q_k+p_k-2],$ we obtain that for a Sturmian word $x$ the number of closed factors of $x$ of each length $n$ is given by
\begin{equation}f^c_x(n)= \sum_{\substack{k\in \omega\\n\in P_k}} (d(n,P_k)+1) + \sum_{\substack{k\in \omega\\n\in Q_k}} (d(n,Q_k)+1).\end{equation}

\begin{example} \rm Consider the Fibonacci word \[x=abaababaabaababaa\cdots\] fixed by the morphism $a\mapsto ab,$ $b\mapsto a.$ Then $p_k=F_k$ and $q_k=F_{k+1}$ where the sequence $(F_k)_{k\in \omega}$ is the Fibonacci sequence given by $F_0=F_1=1$ and $F_{k+1}=F_k+F_{k-1}$ for $k\geq 1.$

Table~\ref{tab:fib} shows the number of closed factors of length $n\leq15$ in the Fibonacci word computed using  (4).

\begin{table}[H]
    \centering
     \caption{\it The number of closed factors in the Fibonacci word.}
    \begin{tabular}{|c|c|c|c|c|c|c|c|c|c|c|c|c|c|c|c|}
\hline
 \textbf n  &\textbf 1 &\textbf 2 &\textbf 3 &\textbf 4 &\textbf 5 & \textbf 6 & \textbf 7 & \textbf 8 & \textbf 9 & \textbf{10} & \textbf{11} & \textbf{12} & \textbf{13} & \textbf{14} & \textbf{15} \\
 \hline
$ \boldsymbol{f_{x}^c(n)}$ & 2 & 1 & 2 & 3 & 4 & 3 & 4 & 5 & 6 & 5 & 6 & 7 & 8 & 9 & 10\\
 \hline
 \end{tabular}
    \label{tab:fib}
\end{table}

For example, for $n=11$ we must determine those $k$ for which either $11\in P_k$ or $11\in Q_k.$
It is easily checked that $11$ only belongs to $P_4=[8,16],$ $Q_3=[7,11]$ and $Q_4=[11,19].$
So\[f^c_{x}(11)=d(11,P_4)+1 + d(11, Q_3)+1 + d(11,Q_4)+1=4+1+1=6.\]

The graph of the function $f_{x}^c$ is shown in Figure~\ref{fig:Fib}. The function is clearly not monotone. 
\end{example}
\begin{figure}[H]
\centering
{\includegraphics[width=1\linewidth]{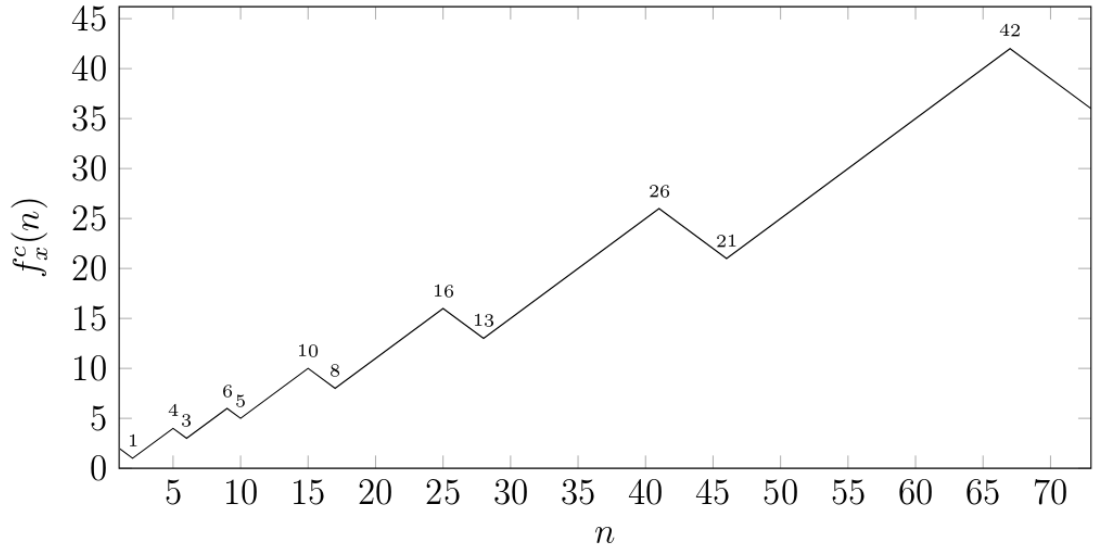}}
\caption{The number of closed factors in the Fibonacci word.}
\label{fig:Fib}
\end{figure}

Figure~\ref{fig:Trib} illustrates the behavior of the number of closed factors of the Tribonacci word $x\in \{a,b,c\}^\nats$ defined as the fixed point of the morphism $a\mapsto ab,$ $b\mapsto ac,$ $c\mapsto a.$

\begin{figure}[H]
\centering
{\includegraphics[width=1\linewidth]{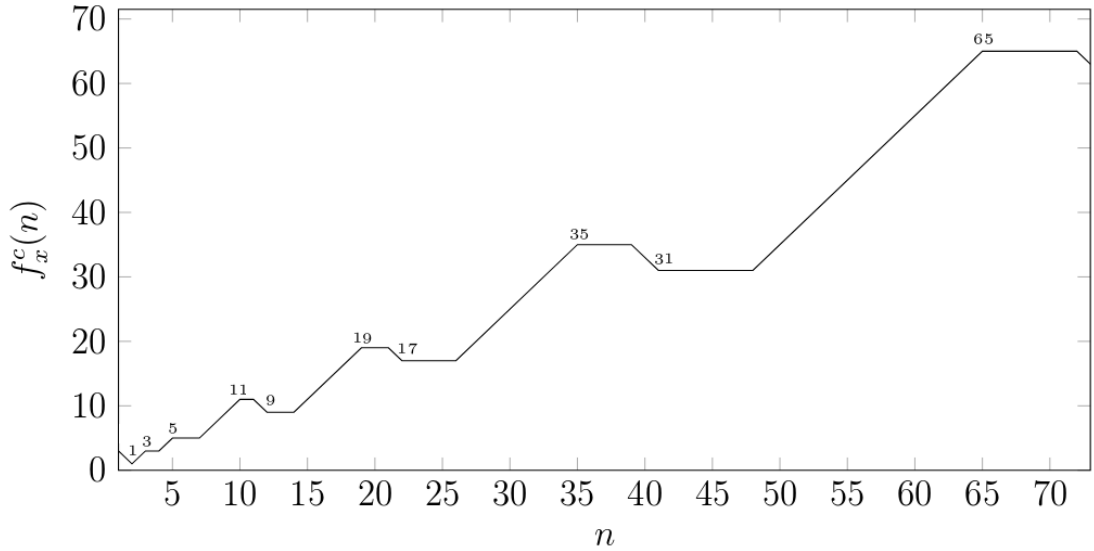}}
\caption{The number of closed factors in the Tribonacci word.}
\label{fig:Trib}
\end{figure}

Our last example (Figure~\ref{fig:x}) illustrates the behavior of the number of closed factors of the Sturmian word $x_r\in \{0,1\}^\nats$ of slope $\alpha$ whose continued fraction expansion begins with $\alpha=[0;3,1,5,4,1,5,1,1,1,2,1,3,2,2,4,\dots]$.
%whose directive sequence begins with\\  $0010000011110111110101101110011000011\cdots$. 

\begin{figure}[H]
\centering
{\includegraphics[width=0.96\linewidth]{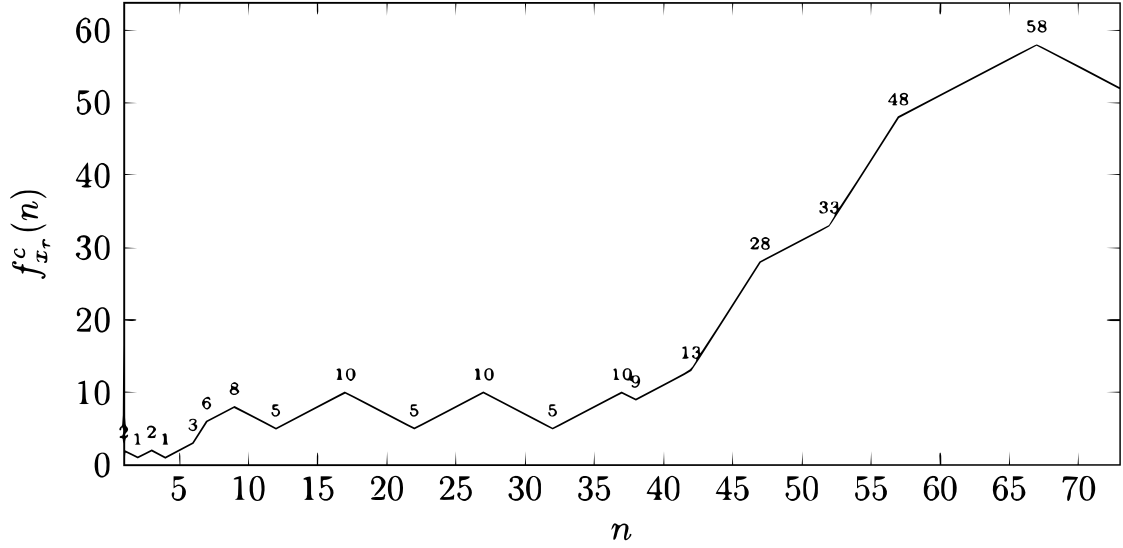}}
\caption{The number of closed factors in the word $x_r$.}
\label{fig:x}
\end{figure}

The above examples suggest that the function $f_{x}^c(n)$ tends to infinity, although it need not be monotone and may contain plateaus and inflection points. Our next result establishes this fact:
\begin{corollary}\label{inf}
If $x\in \A^\nats$ is an Arnoux-Rauzy word, then  
\[
\liminf_{n\rightarrow\infty}{f_x^c(n)}=+\infty.\]
\end{corollary}

\begin{proof}
For each $k\in \omega$ set $I_k=[b_k-p_k+2,b_k+p_k],$ i.e., $I_0=[1,1]$ and $I_k=I_{a_k,k}$ for $k\geq 1.$
Since $p_k\geq 1$ for each $k\in \nats$ it follows that $\nats=\bigcup_{k\in \omega}I_k.$ 
Given $m\in \nats$, pick $j$ such that $p_j-1>2m$, and put $N= b_j+2.$ We will show that $f_x^c(n)\geq m$ for every $n\geq N.$ Notice that since $b_k=b_{k-1}+p_k$, the left hand endpoint of the interval $I_{j+1}$ is $b_j+2.$ Thus for each $n\geq N$ there exists a positive integer $k\geq j$, such that $n \in I_{k+1}.$ We have $|I_{k+1}|=2p_{k+1}-2\geq 2p_j-2>4m$. If $d(n,I_{k+1})\geq m,$ then it follows from Theorem~\ref{main} that $f_x^c(n)\geq m.$ Otherwise we must have either i) $b_{k+1}+p_{k+1}-n<m$ or 
ii) $n-(b_k+2)<m.$ In case i)  we have 

\begin{equation}n-b_{k+1}>p_{k+1}-m\geq p_j-m>m+1.\end{equation} 
Since $m+1\geq 2,$ we have that $n$ also belongs to $I_{k+2}.$ We will show that $d(n,I_{k+2})\geq m.$
By (5) we have $n-(b_{k+1}+2)\geq m.$ Also 
\[b_{k+2}+p_{k+2}-n\geq b_{k+2}+p_{k+2}-(b_{k+1}+p_{k+1})
=2p_{k+2}-p_{k+1}\geq p_{k+2}\geq p_j > 2m+1.\] Thus $d(n,I_{k+2})\geq m$ and hence by Theorem~\ref{main}$f_x^c(n)\geq m.$ 

In case ii) $n < b_k+2 +m\leq b_k+2m+1<b_k+p_k,$ and hence $n\in I_k.$ We will show that $d(n,I_k)\geq m.$ In fact, $b_k+p_k-n>p_k-2-m\geq p_j-2-m\geq m.$  Moreover, since $n\in I_{k+1},$ we have that $n-(b_{k-1}+2) \geq b_k+2-(b_{k-1}+2)=p_k\geq p_j> 2m+1$, and thus $d(n,I_k)\geq m.$
\end{proof}

While the previous result applies to Arnoux-Rauzy words, for a general aperiodic word $x$ the limit inferior of the function $f_x^c(n)$ need not be infinite. For example, in the case of the regular paperfolding word one has that $\liminf_{n\rightarrow\infty} f_x^c(n)=0.$  In fact, in \cite{SS} the authors exhibit an $11$-state automaton which accepts the base $2$ representation of those $n$ for which there is a closed factor of the regular paperfolding word of length $n$ (see Figure 1 in \cite{SS}). As another perhaps simpler example, let $x$ be the fixed point beginning in $a$ of the $2$-uniform morphism $\varphi$  
on the alphabet $\{a,b,c,d\}$ given by $\varphi: a\mapsto ac, b\mapsto ad, c\mapsto bc, d\mapsto bd.$ Then it is easily shown that all factors of $x$ of length $2^n$ ($n\in \nats)$ are open.  We remark that this last example is closely related to the regular paperfolding word. In fact, the regular paperfolding word is the image of the fixed point of $\varphi$ under the mapping which sends $a,c$ to $0$ and $b,d$ to $1.$ 

\section{Open Sturmian words}

\noindent Let  $\A$ be a nonempty set and $w=w_1w_2\cdots w_n\in \A^+$ a word of length $n$ over the alphabet $\A.$ Let $A_w=\{1\leq i\leq n\,|\, w_1\cdots w_i\,\,\mbox{is closed}\} $ and similarly $A'_w=\{1\leq i\leq n\,|\, w_i\cdots w_n\,\,\mbox{is closed}\}.$ Also let $B_w=\{1,2,\ldots,n\}\setminus A_w$ and $B'_w=\{1,2,\ldots,n\}\setminus A'_w.$ Thus $|A_w|$ (resp. $|A'_w|)$ is the number of closed prefixes (resp. suffixes) of $w$ while $|B_w|$ (resp. $|B'_w|)$ is the number of open prefixes (resp. suffixes) of $w.$ Following notation used  in \cite{Fici0} we let $H_w$ (resp. $K_w)$ denote the length of the shortest unrepeated prefix (resp. suffix) of $w.$ Therefore $|A_w|=H_w$ and $|A'_w|=K_w.$ 

\begin{lemma}\label{-A} Let $w=w_1w_2\cdots w_n\in \A^+$  be any open word. Assume that $|A_w|=|B'_w|.$ Then $w$ may be factored as a concatenation of two closed words. \end{lemma}

\begin{proof} Since $w$ is open it follows that $n\notin A_w$ and hence $A_w+1\subseteq \{2,3,\ldots,n\}$ where $A_w+1=\{i+1\,|\, i\in A_w\}.$ Again since $w$ is open we have that $1\in B'_w$ whence $A_w+1\neq B'_w.$
Finally $|A_w+1|=|A_w|=|B'_w|.$  Thus $A_w+1\nsubseteq B'_w$ and hence $A_w+1\cap A'_w\neq \emptyset.$ Pick $i\in A_w$ such that $i+1\in A'_w.$ Then $w_1\cdots w_i$ is a closed prefix of $w$ while $w_{i+1}\cdots w_n$ is a closed suffix of $w.$ \end{proof}

We note that if $w$ is a concatenation of two closed words, then it need not be the case that $|A_w|=|B'_w|.$ For example, $w=aababbab$ is the concatenation of the closed words $aa$ and $babbab$ while $|A_w|=2$ and $|B'_w|=4.$

Again borrowing from the notation in \cite{Fici0}, let $S^l(w)$ (resp. $S^r(w))$ denote the number of distinct left (resp. right) special factors of $w.$ As usual we consider the empty word $\varepsilon$ to be both left and right special. In what follows we restrict to binary words over the alphabet $\{a,b\}.$ We recall the following lemma from \cite{Fici0}:

\begin{lemma}[Lemma 6.2 in \cite{Fici0}]\label{A} Let $w\in \{a,b\}^+$ be any binary word. Then $S^l(w)=|w|-H_w$ and $S^r(w)=|w|-K_w.$
\end{lemma}

\begin{lemma}\label{B} Let $w\in \{a,b\}^+$ be an open Sturmian word. Then $S^l(w)=K_w$ and $S^r(w)=H_w.$
\end{lemma}

\begin{proof} It suffices to prove the first assertion as the second follows from the first by taking reversals.  Let $u$ denote the shortest unrepeated suffix of $w$ so that $K_w=|u|.$ Without loss of generality we may assume $u$ begins in $a.$ If $u=a$ then $w=b^{n-1}a$ and hence $\varepsilon$ is the unique left special factor of $w$ and so $K_w=1=S^l(w)$ as required.  Otherwise we may write $u=av$ for some $v\in\{a,b\}^+.$ As $w$ is open, $v$ must occur internally in $w.$ Furthermore as $av$ only occurs in $w$ as a suffix it follows that any internal occurrence of $v$ in $w$ must be preceded by the letter $b.$ Hence $v$ is a left special factor of $w$ and hence so is every prefix of $v.$ Thus $K_w\leq S^l(w).$ We now claim that the only left special factors of $w$ are $\varepsilon$ and all prefixes of $v,$ i.e., $K_w=S^l(w).$  In fact, let $v'$ be a left special factor of $w.$ Then both $av'$ and $bv'$ occur in $w.$ If $v'$ is not a prefix of $v,$ then as $w$ is Sturmian it follows that $v$ is a proper prefix of $v'$ which would imply that $av$ occurs internally in $w,$ a contradiction. \end{proof}

\noindent As an immediate consequence of the above two lemmas we obtain :

\begin{corollary}\label{C} Let $w\in \{a,b\}^+$ be an open Sturmian word. Then $|w|=H_w+K_w$ or equivalently $w=uv$ where $u$ (resp. $v)$ is the shortest unrepeated prefix (resp. suffix) of $w.$
\end{corollary}

\begin{corollary}\label{D} Let $w=w_1w_2\cdots w_n\in \{a,b\}^+$  be an open Sturmian word. Then $|A_w|=|B'_w|$ and $|A'_w|=|B_w|.$ In other words, the number of closed prefixes of $w$ is equal to the number of open suffixes of $w.$
\end{corollary}

\begin{proof} Using Corollary~\ref{C} we have that $|A_w|=H_w=|w|-K_w=n-|A'_w|=|B'_w|.$ \end{proof}

\noindent Combining Lemma~\ref{-A} with Corollary~\ref{D} yields the following result first announced in \cite{Fici} 

\begin{corollary} Every open Sturmian word $w\in \{a,b\}^+$ is a concatenation of two closed words.  \end{corollary}


\begin{thebibliography}{99}



\bibitem{ABCD} J.-P. Allouche, M. Baake, J. Cassaigne, D. Damanik, Palindrome complexity, Selected papers in honor of Jean Berstel, Theoret. Comput. Sci., 292 (2003), pp. 9--31.

\bibitem{AR} P. Arnoux, G. Rauzy, Repr\'{e}sentation g\'{e}om\'{e}trique de suites de complexit\'{e} $2n+1,$ Bull. Soc. Math. France 119 (2) (1991), pp. 199--215.

\bibitem{Ber} J. Berstel, Sturmian and episturmian words (A survey of some recent results), in: Proceedings of CAI 2007, Lecture Notes in Computer Science, vol. 4728, 2007, pp. 23--47.

% \bibitem{full2} S. Brlek, S. Hamel, M. Nivat, C. Reutenauer, On the palindromic complexity of infinite words,
Internat. J. Found. Comput. Sci., 15 (2004), pp. 293-306.

\bibitem{BDF} M. Bucci, A. De Luca, G. Fici,  Enumeration and structure of trapezoidal words, Theoret. Comput. Sci. 468 (2013), pp. 12--22.


\bibitem{CdL} A. Carpi, A. de Luca, Periodic-like words, periodicity and boxes, Acta. Inform. 37 (2001), pp. 597--618.

\bibitem{CFSZ} J. Cassaigne, G. Fici, M. Sciortino, L.Q. Zamboni, Cyclic complexity of words, J. Combin. Theory Ser. A, 145 (2017), pp. 36--56.

\bibitem{CPZ} E. Charlier, S. Puzynina, L.Q. Zamboni, On a group theoretic generalization of the Morse-Hedlund theorem, Proceedings of the AMS, 145 (2017), pp. 3381--3394.

\bibitem{CovHed} E.~M. Coven, G.~A. Hedlund, Sequences with minimal block growth,
 Math. Systems Theory 7 (1973), pp. 138--153.
 
 \bibitem{deL} A. de Luca, Sturmian words: structure, combinatorics and their arithmetics, Theoret. Comput. Sci. 183 (1997), pp. 45--82. 
 
%  \bibitem{FWThSt} A. de Luca, F. Mignosi, Some combinatorial properties of Sturmian words. Theoret. Comput. Sci., 136 (1994), pp. 361--385
 
 \bibitem{DJP} X. Droubay, J. Justin, G. Pirillo, Episturmian words and some constructions of de Luca and Rauzy, Theoret. Comput. Sci. 255 (2001), pp. 539--553.
 
  \bibitem{Fici0} G. Fici, Special factors and the combinatorics of suffix and factor automata, Theoret. Comput. Sci. 412 (2011), p. 3604--3615.

 \bibitem{Fici} G. Fici, Open and closed words,  Bull. Eur. Assoc. Theor. Comput. Sci. EATCS No. 123 (2017), p. 138--147.

 \bibitem{GJWZ} A. Glen, J. Justin, S. Widmer, L.Q. Zamboni,  Palindromic richness, European J. Combin. 30 (2009), no. 2, pp. 510--531.


% \bibitem{J} J. Justin, On a paper by Castelli, Mignosi, Restivo, RAIRO: Theoret.Informatics Appl. 34 (2000), pp. 373--377.

\bibitem{JP} J. Justin, G. Pirillo, Episturmian words and episturmian morphisms, Theoret. Comput. Sci. 276 (2002), pp. 281--313.

\bibitem{KaZa} T. Kamae, L.Q. Zamboni,  Sequence entropy and the maximal pattern complexity of infinite words, Ergodic Theory Dynam. Systems  22 (2002),  pp. 1191--1199.

% \bibitem{loth} M. Lothaire, Combinatorics on Words (Addison-Wesley, Reading, MA, 1983)

\bibitem{MoHe1}
M. Morse, G. Hedlund, Symbolic dynamics,  Amer. J. Math.  60 (1938), pp. 815--866.

\bibitem{Priv} J. Peltom\"aki, Introducing privileged words: privileged complexity of Sturmian words, Theoret. Comput. Sci. 500 (2013), pp. 57--67. 


\bibitem{RSZ} G. Richomme, K. Saari, L.Q. Zamboni, Abelian complexity of minimal subshifts,  J. Lond. Math. Soc. (2), 83 (2011),  pp. 79--95.

% \bibitem{hat} R. Risley, L. Zamboni, A generalization of Sturmian sequences: Combinatorial structure and transcendence, Acta Arithmetica 95.2 (2000), pp. 167--184. 

\bibitem{SS} L. Schaeffer, J. Shallit, Closed, palindromic, rich, privileged, trapezoidal, and balanced words in automatic sequences,  Electron. J. Comb.,
23(1) (2016), pp. 1--25.


\bibitem{TZ} R. Tijdeman, L.Q. Zamboni, Fine and Wilf words for any periods, Indag. Math. (N.S.) 14 (2003), pp. 135--147.

\end{thebibliography}
\end{document}